\documentclass[11pt,a4paper]{article}

\usepackage{amsmath,amsfonts,amssymb}
\usepackage{color, graphicx, soul}
\usepackage{mathtools}
\usepackage{xargs}
\usepackage{enumerate}

\usepackage[colorinlistoftodos,prependcaption,textsize=footnotesize]{todonotes}
\newcommandx{\attn}[2][1=]{\todo[linecolor=red,backgroundcolor=blue!25,bordercolor=red,#1]{#2}}
\newcommandx{\other}[2][1=]{\todo[linecolor=OliveGreen,backgroundcolor=OliveGreen!25,bordercolor=OliveGreen,#1]{#2}}
\newcommandx{\thiswillnotshow}[2][1=]{\todo[disable,#1]{#2}}
\usepackage{hyperref}
\hypersetup{
	colorlinks,
	linkcolor={black},
	citecolor={blue!50!black},
	urlcolor={blue!80!black}
}

\newtheorem{theorem}{Theorem}[section]
\newtheorem{lemma}[theorem]{Lemma}

\newtheorem{proposition}[theorem]{Proposition}
\newtheorem{remark}[theorem]{Remark}
\newtheorem{example}[theorem]{Example}

\newcommand{\beq}{\begin{equation}}
\newcommand{\eeq}{\end{equation}}
\newcommand{\beqa}{\begin{eqnarray}}
\newcommand{\eeqa}{\end{eqnarray}}
\newcommand{\beqas}{\begin{eqnarray*}}
\newcommand{\eeqas}{\end{eqnarray*}}
\newcommand{\bi}{\begin{itemize}}
\newcommand{\ei}{\end{itemize}}

\newcommand{\PSDcone}[1]{{\mathcal{S}^{#1}_+}}
\renewcommand{\S}{\mathcal{S}}      
\renewcommand{\Re}{\mathbb{R}}
\def\endproof{{\ \hfill\hbox{%
      \vrule width1.0ex height1.0ex
    }\parfillskip 0pt}\par}

\def\qed{\ifhmode\unskip\nobreak\fi\ifmmode\ifinner\else\hskip5pt\fi\fi
  \hbox{\hskip5pt\vrule width5pt height5pt depth1.5pt\hskip1pt}}

\def\QED{\ifhmode\unskip\nobreak\fi\ifmmode\ifinner\else\hskip5pt\fi\fi
  \hbox{\hskip5pt\vrule width5pt height5pt depth1.5pt\hskip1pt}}
\newcommand{\ambSpace}{\ensuremath{\mathcal{E}}}
\newcommand{\stdCone}{ \ensuremath{\mathcal{K}}}
\newcommand{\stdFace}{ \ensuremath{\mathcal{F}}}
\newcommand{\stdAff}{ \ensuremath{\mathcal{V}}}
\newcommand{\reInt}{\ensuremath{\mathrm{ri}\,}}
\newcommand{\inProd}[2]{\langle #1 , #2 \rangle }

\numberwithin{equation}{section}

\usepackage{cite}

\renewcommand{\S}{\mathcal{S}} 
\renewcommand{\Re}{\mathbb{R}} 

\def\av{{v_a}}
\newcommand{\norm}[1]{\lVert{#1}\rVert}
\newcommand{\vJ}{\bar v(J)}
\begin{document}

\title{
Closing Duality Gaps of SDPs through Perturbation\thanks{The first author is supported in part by JSPS Grant-in-Aid for Scientific Research (B) 21H03398,
	the second author is supported in part by the same grant and JSPS Grant-in-Aid for Young Scientists JP19K20217,
	the third author is supported in part by JSPS Grant-in-Aid for Scientific Research (C) JP17K00031 and
	the same grant for Scientific Research (B) JP20H04145, and 
	the fourth author is supported in part by JSPS Grant-in-Aid for Young Scientists JP20K19748 and the same grant for Scientific Research (B)20H04145.
}  \\ \ \ \ \\
}


\author{Takashi Tsuchiya\footnote{National Graduate Institute for Policy Studies, 7-22-1 Roppongi, Minato-ku, Tokyo 106-8677 Japan, 
              e-mail: {tsuchiya@grips.ac.jp}}           %
\and
         Bruno F.~Louren\c{c}o\footnote{The Institute of Statistical Mathematics, Midori-cho 10-3, Tachikawa, 190-8562 Tokyo Japan, e-mail: {bruno@ism.ac.jp}} \and Masakazu Muramatsu\footnote{The University of Electro-Communications, 
           1-5-1 Chofugaoka, Chofu, Tokyo 182-8585 Japan, e-mail: MasakazuMuramatsu@uec.ac.jp } \and Takayuki Okuno\footnote{
Faculty of Science and Technology Department of Science and Technology, Seikei University, 3-3-1, Kita, Kichijoji, Musashino, Tokyo, 180-8633, Japan; 
Center for Advanced Intelligence Project, RIKEN, 1-4-1 Nihonbashi, Chuo-ku, Tokyo 103-0027 Japan.
email: {takayuki-okuno@st.seikei.ac.jp}       }
}

\date{April 2023}

\maketitle
\begin{abstract}
Let $({\bf P},{\bf D})$ be a primal-dual pair of SDPs with a nonzero finite duality gap. Under such circumstances, ${\bf P}$ and ${\bf D}$ are weakly feasible and if we perturb the problem data to recover strong feasibility, the (common) optimal value function $v$ as a function of the perturbation is not well-defined at zero (unperturbed data) since there
are ``two different optimal values'' $v({\bf P})$ and $v({\bf D})$, where $v({\bf P})$ and $v({\bf D})$ are the optimal values of ${\bf P}$ and ${\bf D}$ respectively.   Thus, continuity of $v$ is lost at zero though $v$ is continuous elsewhere.
Nevertheless, we show that a limiting version $\av$ of $v$ is a well-defined monotone decreasing continuous bijective function connecting $v({\bf P})$ and $v({\bf D})$ with domain $[0, \pi/2]$ under the assumption that both ${\bf P}$ and ${\bf D}$ have singularity degree one. The domain $[0,\pi/2]$ corresponds to directions of perturbation defined in a certain manner.
Thus, $\av$ ``completely fills'' the nonzero duality gap under a mild regularity condition. Our result is tight in that there exists an instance with singularity degree two for which $\av$ is not continuous.

\end{abstract}

\baselineskip=15pt
\section{Introduction}

Consider the standard form dual pair of semidefinite programs:
\begin{align}
\ \ \ \min_{X} \ C \bullet X \ \ &{\rm{s.t.}}\ A^i \bullet X = b_i,\ i=1, \ldots, m, X\succeq 0 \tag{{\bf{P}}} \label{eq:p}\\
\ \ \ \max_{y,S} \ b^T y \ \ \ \ &{\rm{s.t.}} \ C - \sum_{i=1}^m A^i y_i  = S,\ S\succeq 0, \tag{{\bf{D}}}\label{eq:d}
\end{align} 
where $C$, $A^i, i = 1, \ldots, m$, $X$, and $S$ are real symmetric $n\times n$ matrices and $y \in \Re^m$.    
We denote the optimal values of {\bf P} and {\bf D} by $v({\bf P})$ and $v({\bf D})$, respectively.
We use analogous notation throughout this paper to represent optimal values.
We assume that {\bf P} and {\bf D} are feasible but not necessarily strongly feasible, i.e., neither {\bf P} nor {\bf D} satisfy Slater's condition.
Under this assumption, {\bf P} and {\bf D} may have a finite nonzero duality gap as shown in 
the following famous example adapted from Ramana \cite{Ramana95anexact}:

\begin{example}[\!\!{\cite[Example~4]{Ramana95anexact}}]\label{ex:gap}
The problem {\bf D} is
\[
\max\ y_1\ {\rm{s.t.}}\ \ \left(\begin{array}{ccc} 1 -y_1 & 0 & 0 \\
0 & -y_2 & -y_1 \\ 
0 & -y_1 & 0
\end{array}\right) \succeq 0.
\]
With that, we have
\[
C=
\left(\begin{array}{ccc} 1 & 0 & 0 \\
0 & 0 & 0 \\ 
0 & 0 & 0
\end{array}\right),\ \ \ 
A^1=
\left(\begin{array}{ccc} 1 & 0 & 0 \\
0 & 0 & 1 \\ 
0 & 1 & 0
\end{array}\right),\ \ \ 
A^2=
\left(\begin{array}{ccc} 0 & 0 & 0 \\
0 & 1 & 0 \\ 
0 & 0 & 0
\end{array}\right),\ \ \ b_1=1.
\]
We have $v({\bf D})= 0$ for this problem, because $y_1=0$ is the only possible value for the lower-right $2\times 2$ submatrix to be 
positive semidefinite.

The associated primal {\bf P} is
\[
\min\ x_{11}\ \ {\rm{s.t.}}\ x_{11}+ 2 x_{23} = 1, \ x_{22}=0,\ \left(\begin{array}{ccc} x_{11} & x_{12} & x_{13}\\ x_{12} & x_{22} & x_{23}\\
x_{13} & x_{23} & x_{33}\end{array}\right)\succeq 0.
\]
We have $v({\bf P})=1$ for this problem, because $x_{23}=0$ must hold for positive semidefiniteness of the lower-right $2\times 2$
submatrix, which drives $x_{11}$ to be 1.
\end{example}

%
%
%

In general, \ref{eq:p} and \ref{eq:d} might fail to have interior feasible solutions and this opens the possibility of nonzero duality gaps as in Example~\ref{ex:gap}.
On the other hand, most algorithms for SDP are constructed under the assumption that {\bf P} and {\bf D} have 
interior feasible solutions.  This leads naturally to the question of how to solve SDPs with  finite nonzero duality gaps with the usual algorithms.
A simple approach to fix that is to perturb the problem data so that constraint qualifications are satisfied \cite{ipm:Potra16,sremac2017complete,TLMO2019}. We can consider, for example, 
the following perturbed/regularized system ${\bf P}(\varepsilon,\eta)$ and ${\bf D}(\varepsilon,\eta)$ from \cite{TLMO2019}. 


\medskip\noindent
{\bf Regularized Primal-Dual Standard Form SDP (RPD-SDP)}
\begin{equation}\label{Pptbd}
{\bf P}(\varepsilon,\eta):\ \ \ \min_{X} \ (C + \varepsilon I)\bullet X \ \ {\rm{s.t.}}\ A^i \bullet X = b_i+ \eta A^i\bullet I,\ i=1, \ldots, m,\ X\succeq 0
\end{equation}
and 
\begin{equation}\label{Dptbd}
{\bf D}(\varepsilon,\eta):\ \ \ \max_{y,S} \sum_{i=1}^m (b_i + \eta A^i\bullet I)y_i \ \ {\rm{s.t.}} \ C - \sum_{i=1}^m A^i y_i + \varepsilon I = S,\ \ \ S\succeq 0,
\end{equation}
where $I$ denotes the $n\times n$ identity matrix.
We call the pair \eqref{Pptbd} and \eqref{Dptbd} the {\em Regularized Primal-Dual Standard Form SDP or 
RPD-SDP for short}.
${\bf P}(\varepsilon,\eta)$ and ${\bf D}(\varepsilon,\eta)$ reduce to
{\bf P} and {\bf D} when $\varepsilon$ and $\eta$ are set to zero.
RPD-SDP  is obtained by relaxing the semidefinite constraints $X\succeq 0$ of {\bf P} and $S\succeq 0$ of 
{\bf D} to $X\succeq -\eta I$ and $S\succeq -\varepsilon I$, respectively, and 
by redefining $X:=X+\eta I$ and $S:=S+\varepsilon I$. 

Under the assumption that {\bf P} and {\bf D} are feasible, the perturbed problems admit interior feasible solutions for any $\varepsilon>0$
and $\eta>0$, so, in this sense, the lack of interior solutions of {\bf P} and {\bf D} is fixed. However, this is only useful if something can be said about how the optimal values of ${\bf D}(\epsilon,\eta)$ and ${\bf P}(\epsilon,\eta)$ relate to the optimal values of the original ${\bf P}$ and ${\bf D}$, so let us briefly examine this issue.

For $\varepsilon > 0$ and $\eta > 0$, 
Slater's condition is satisfied at both  \eqref{Pptbd} and \eqref{Dptbd}, so they have optimal solutions and a 
common optimal value, which we denote by  $v(\varepsilon,\eta)$.
In the sequel, $v(\cdot,\cdot)$ is referred to as {\em pd-regularized optimal value function}. 
We note that 
$v(\varepsilon,0)$ and $v(0,\eta)$ are also well-defined for any $\varepsilon>0$
and $\eta > 0$, since ${\bf P}(\varepsilon,0)$ and ${\bf D}(0,\eta)$ have interior feasible solution for any $\varepsilon>0$ and $\eta>0$ and there is no duality gap in these cases according to the standard duality theory 
for convex programming.

But $v(0,0)$ is different since it is not well-defined when there exists a finite nonzero duality gap, and ironically,
{\em the value of $v(0,0)$} is what we really wish to compute.
Thus, the perturbation/regularization approach might not be theoretically sound when there are nonzero duality gaps.

That said, we have recently analyzed the behavior of  the pd-regularized optimal value function $v(\varepsilon,\eta)$ in the neighbourhood of 
$(\varepsilon,\eta)=(0,0)$ and demonstrated that
$v(\varepsilon,\eta)$ have a directional limit when approaching $(0,0)$, see \cite{TLMO2019}. 
Let us define the directional limit
\begin{equation}\label{eq:av}
\av(\theta) \coloneqq \lim_{t\downarrow 0}v(t\cos\theta,t\sin\theta).
\end{equation}
The function $\av$ is referred to as {\em limiting pd-regularized optimal value function}.
Then, the following theorem holds.
\begin{theorem}[\!\!{\cite[Theorem~2]{TLMO2019}}]\label{thm:tlmo}
If {\bf~P} and {\bf~D} are feasible, the limiting pd-regularized optimal value function $\av(\theta)$ has the following properties.
\begin{enumerate}
\item $\av(0)=v({\bf P})$, $\av(\pi/2)=v({\bf D}).$
\item  $\av(\theta)$ is monotone decreasing in $[0,\pi/2]$ and is continuous on $(0,\pi/2)$.
\end{enumerate}
\end{theorem}

In the case of Example~\ref{ex:gap}, $\av(\theta)$ is written as follows:
\begin{equation}\label{va}
\av(\theta)=
\begin{cases}1-\tan\theta &\hbox{if}\ 0 \leq \tan\theta\leq\frac12,\\ 
\frac14\cot\theta &\hbox{if}\ 1/2\leq \tan\theta. 
\end{cases}
\end{equation}
See  \cite[Example 1]{TLMO2019} for its derivation\footnote{The formula obtained in \cite{TLMO2019} is slightly different than \eqref{va}.  But we obtain exactly the same formula by letting $\alpha=\cos\theta$ and $\beta=\sin\theta$.}.
In this instance, 
it can be shown that $\av$ is a strictly monotone decreasing continuous function on the closed interval $[0, \pi/2]$.
This implies that, for each value $w \in [v({\bf D}), v({\bf P})]$,
there exists a unique angle $\theta\in [0, \pi/2]$ such that $\av(\theta)=w$ holds.
Thus, the limiting pd-regularized optimal value function $\av$ is a continuous bijection from $[0,\pi/2]$ to 
$[v({\bf D}), v({\bf P})]$.  
This shows that the gap is only superficial and hidden continuity behind the gap was brought to the surface. 
We may say that the nonzero duality gap is ``filled completely'' because we may reach any value in $[v({\bf D}), v({\bf P})]$ by appropriately selecting $\theta$.
It would be nice if this structure existed for any SDP with a finite nonzero duality gap. 

Theorem~\ref{thm:tlmo}, however, does not exclude the
possibility that $\av$ is discontinuous on the boundary of domain at $\theta=0$ and $\theta=\pi/2$.
In particular, the continuity of  $\av(\theta)$ over the entire interval $[0,\pi/2]$ 
was left open in \cite{TLMO2019}.
If $\av$ turns out to be continuous on the both end points $\theta=0$ and
$\theta=\pi/2$, then,  $\av$ is a continuous function from $[0,\pi/2]$ to 
$[v({\bf D}), v({\bf P})]$, and 
as a consequence, by controlling the direction to which $(t\cos \theta, t \sin \theta)$ approaches $(0,0)$ we may force $\av(\theta)$ to assume any value in $ [v({\bf D}), v({\bf P})]$. 

Now we are ready to state our goal in this paper.  We presents two results under the assumption of a finite nonzero duality gap, 
one is positive and the other negative.
First, we will show that  $\av$ is continuous at $\theta=0$ and $\theta=\pi/2$ if the \emph{singularity degree} of {\bf D} and {\bf P} are both one, and consequently, the limiting pd-regularized optimal value function $\av$ is a continuous 
bijective function from $[0,\pi/2]$ to 
$[v({\bf D}), v({\bf P})]$ in that case.  Here, the singularity degree for a feasible SDP is defined as the
minimum number of facial reduction steps necessary to regularize the problem (see Section~\ref{sec:pre}). 
Then, second, we present an instance where $\av$ is discontinuous at $\theta = \pi/2$, thus showing that continuity cannot be expected to hold in general.

This paper is organized as follows.
In Section~\ref{sec:pre}, we discuss the notation and some prerequisite notions.
In Section~\ref{sec:main}, we analyze
the continuity of $\av$ under the assumption that the problems have singularity degree one.
In Section~\ref{sec:counter_ex} we present an instance to show that $\av$ may be discontinuous in general. 
Section~\ref{sec:conc} presents some concluding remarks.

\section{Preliminaries}\label{sec:pre}
In this section, we introduce some terminology and mathematical definitions. 
The space of real $p\times p$ symmetric matrices will be denoted by $\S^p$. The space of $p\times q$ real matrices will be denoted by $\Re^{p\times q}$. The cone of semidefinite matrices in $\S^p$ will be denoted by $\PSDcone{p}$. The $p$-dimensional nonnegative orthant will be denoted by $\Re^p_+$.
In this paper, we are using the convention that whenever the identity matrix $I$ appears in a expression it has the ``correct size'' as to make the expression well-defined.

We recall that \ref{eq:p} is said to be \emph{strongly feasible}
if there exists a positive definite feasible solution $X$ to \ref{eq:p}. This is the same as saying that \ref{eq:p} satisfies \emph{Slater's condition}.  \ref{eq:p} is said to be \emph{weakly feasible} if it is feasible but not strongly feasible.
Similar definitions apply to \ref{eq:d}. 

In semidefinite programming it sometimes happens 
that the optimal value of \ref{eq:p} or \ref{eq:d} is finite, but no optimal solution exists\footnote{Consider, for example, $\min x_{11}$ s.t. $ \begin{psmallmatrix}
	x_{11} & 1 \\ 1 & x_{22}
	\end{psmallmatrix} \in \PSDcone{2} $. The optimal value is $0$ but not attained.}.
In any case, provided that \ref{eq:p} is feasible,  there exists a sequence of matrices $\{X^k\}$ such that the $X^k$ are feasible to \ref{eq:p} and $C\bullet X^k$ converges to $v({\bf P})$.
We say that such a sequence is {\em an optimal sequence} (for \ref{eq:p}).
Analogously, an optimal sequence for \ref{eq:d} is a sequence $\{(y^k,S^k)\}$ for which  $C - \sum_{i=1}^m A^i y_i^k  = S^k,\ S^k\succeq 0$ and $b^Ty^k \to v({\bf D})$ holds.

\paragraph{Singularity Degree}
Let $\stdCone \subseteq \ambSpace$ be a closed convex cone contained in a finite-dimensional Euclidean space $\ambSpace$ equipped with an inner product $\inProd{\cdot}{\cdot}$.
For a convex set ${\cal C} \subseteq \ambSpace$, let $\reInt {\cal C}$ denote its relative interior. Also, let ${\cal C}^*$ and ${\cal C}^\perp$ denote the dual cone and orthogonal complement of ${\cal C}$ with respect to $\inProd{\cdot}{\cdot}$, respectively.

Let $\stdAff \subseteq \ambSpace$ be an affine space such that $\stdAff \cap \stdCone \neq \emptyset$ and consider the following feasibility problem:
\begin{equation}\label{eq:feas}\tag{Feas}
\text{find} \quad x \in \stdAff \cap \stdCone.
\end{equation}
In this case, there exists an unique \emph{minimal face} $\stdFace$ of $\stdCone$ with the property that $(\stdAff \cap \stdCone) \subseteq  \stdFace$ and $\stdAff \cap (\reInt \stdFace) \neq \emptyset$. When \eqref{eq:feas} satisfies Slater's condition (i.e., $\stdAff \cap (\reInt\stdCone) \neq \emptyset$), then the face $\stdFace$ is $\stdCone$ itself.

The process of finding the minimal face $\stdFace$ is known as \emph{facial reduction} and is one of the standard approaches for handling conic linear programs that fail to satisfy constraint qualifications, e.g., \cite{borwein_facial_1981,sturm_error_2000,WM13,pataki_strong_2013,LMT15,LP17,DW17}.
While the original problem might suffer from pathologies arising from a lack of constraint qualifications, identifying $\stdFace$ makes it possible to construct a new equivalent problem where the absence of constraint qualifications is fixed.

The basic facial reduction algorithm as described, say, in \cite{WM13}, is based on the following observation. We have 
$\stdAff \cap (\reInt \stdCone)  = \emptyset$ (i.e., Slater's condition fails) if and only if 
$\stdAff$ and $\reInt \stdCone$ can be \emph{properly separated} by a hyperplane that does not contain $\reInt \stdCone$ entirely, see \cite[Theorem~20.2]{RT74}.
Under the hypothesis that $\stdAff \cap \stdCone \neq \emptyset$, this implies the existence of $s \in (\stdAff^\perp\cap \stdCone^*) \setminus \stdCone^\perp$, e.g., see \cite[Lemma~3.2]{WM13}. Then, letting 
\[
\stdFace \coloneqq \stdCone \cap \{s\}^\perp,
\]
$\stdFace$ is a face of $\stdCone$ strictly contained in $\stdCone$ (since $s \not \in \stdCone^\perp$) with the property that $\stdAff \cap \stdCone \subseteq \stdFace$. 
If it turns out that $\stdAff \cap \reInt \stdFace \neq \emptyset$, then we are done. Otherwise, we can apply the same separation result again to $\stdAff$ and $\stdFace$ and repeat the process. 
Since the dimension of faces decrease at each step, this process must end in a finite number of steps, e.g., \cite[Theorem~3.2]{WM13}. This leads to a chain of $\ell$ faces of $\stdCone$
\[
\stdFace _{\ell}  \subsetneq \cdots \subsetneq \stdFace_1 = \stdCone\]
and $\ell-1$ reducing directions $\{s_1, \ldots, s_{\ell-1}\}$ such that:
\begin{enumerate}[{\rm (i)}]
	\item \label{prop:fra1} for all $i \in \{1,\ldots, \ell -1\}$, we have
	\begin{flalign}\label{reduce}%
	s_i \in \stdFace _i^* \cap \stdAff^\perp \ \ \ {and}\ \ \
	\stdFace _{i+1} = \stdFace _{i} \cap \{s_i\}^\perp.
	\end{flalign}		
	\item \label{prop:fra2} $\stdFace _{\ell} \cap  \stdAff = \stdCone \cap \stdAff$ and 	$\stdAff \cap (\reInt\stdFace_{\ell}) \neq \emptyset$.
\end{enumerate}
We note that $\ell$ must be bounded above by the dimension of $\stdCone$, see also \cite{LMT15} for tighter bounds. Furthermore, item~\eqref{prop:fra2} implies that $\stdFace_{\ell}$ is  the minimal face of $\stdCone$ containing $\stdCone \cap \stdAff$.

There is freedom in the choice of the $s_i$ in item~\eqref{prop:fra1}, so difference choices of $s_i$'s might lead to a smaller or larger $\ell$.
The \emph{singularity degree} of \eqref{eq:feas} is the minimal number of \emph{facial reduction steps} (i.e., the $\ell$) in order to find the minimal face of \eqref{eq:feas} as in items \eqref{prop:fra1} and \eqref{prop:fra2}. 
The singularity degree was discussed extensively by Sturm in \cite{sturm_error_2000}, although Sturm's definition of singularity degree is slightly different from the most current usage of the term, see \cite[Footnote~3]{LMT15}.

For the main result of this paper, we will only need a discussion of problems having singularity degree one, so we shall focus on that. \eqref{eq:feas} is said to have \emph{singularity degree one} if it is feasible and there exists 
 $s \in (\stdAff^\perp\cap \stdCone^*) \setminus \stdCone^\perp$ such that
 \begin{equation*}
 \stdAff \cap (\reInt (\stdCone \cap \{s\}^\perp)) \neq \emptyset.
 \end{equation*}
 In the following, we discuss what this definition means for the 
problems \ref{eq:p} and \ref{eq:d}. 
Letting $\ambSpace \coloneqq \S^n$, $\stdCone \coloneqq \PSDcone{n}$ and $\stdAff \coloneqq C + \mathcal{L}$, where $\mathcal{L}$ is the span of the $A^i$ in \ref{eq:d}, the problem \eqref{eq:feas} correspond to the feasible ``slacks'' associated to dual problem \ref{eq:d}. Then, \ref{eq:d} has singularity degree one if and only if there exists  a nonzero $X \in \PSDcone{n}$,  with $C\bullet X = 0$ and $A^i \bullet X = 0$ for all $i$ in such a way that there exists $y \in \Re^m$ satisfying
\[
C- \sum _{i=1}^m A^i y_i \in \reInt (\stdCone \cap \{X\}^\perp).
\] 
Before we move further, we need to recall a classical characterization of the faces of $\PSDcone{n}$. Namely, every face $\stdFace$ of $\PSDcone{n}$ is linearly isomorphic to a smaller positive semidefinite cone. Furthermore, there exists $r \leq n$ and a nonsingular matrix $V$ such that 
\begin{equation}\label{eq:faces}
V\stdFace V^T = \left\{ X \in \S^n \mid X = \begin{pmatrix}
Y & 0 \\
0 & 0
\end{pmatrix}, Y \in \PSDcone{r} \right\},
\end{equation}
e.g., see \cite{pataki_handbook} and also \cite[Section~6]{BC75}.

\medskip

Now we are ready to state the following proposition which plays a fundamental role in our analysis.  

\begin{proposition}\label{prop:sd1}
Suppose that \ref{eq:d} is feasible. Then the following statements hold.
\begin{enumerate}
	\item After appropriate rescaling\footnote{Rescaling \ref{eq:d} corresponds to selecting a  nonsingular matrix $V$ and replacing the $C$ and the $A^i$ with $VCV^T$ and $VA^iV^T$. This transformation preserves the set of $y$ that are feasible for \ref{eq:d} and also preserves the presence (or absence) of duality gaps. For the purposes of this paper, there is no difference between analyzing a problem or its rescalings.
	}, there exists $r \leq n$ such that 
	defining $L(y)\coloneqq C -\sum_{i=1}^m A^i y_i$, the following items hold:
	\begin{enumerate}[$(a)$]
		\item 	$L(y) \in \PSDcone{n}$ if and only if 
		\[
		L(y) \in  \left\{ X \in \S^n \mid X = \begin{pmatrix}
		Y & 0 \\
		0 & 0
		\end{pmatrix}, Y \in \PSDcone{r} \right\}.
		\]
		\item (Slater's condition is satisfied for the reduced problem) There exists $y \in \Re^m $ such that $L(y) = \begin{pmatrix}
		Y & 0 \\
		0 & 0
		\end{pmatrix}$, $Y \in \reInt(\PSDcone{r})$ (i.e., $Y \succ 0$).
	\end{enumerate}
	In particular, defining $L_{11}:\Re^m \to \S^r, L_{12}:\Re^m \to \Re^{r\times (n-r)}, L_{22}: \Re^m \to \S^{n-r}$ so that 
	\[
	L(y) = \begin{pmatrix}
	L_{11}(y) & L_{12}(y)\\ L_{12}^T(y) & L_{22}(y)
	\end{pmatrix}, \qquad \forall y \in \Re^m,
	\]
 \ref{eq:d} is equivalent to the following strongly feasible problem
	\[
	{\bf RD}\ \ \ 
	\max_{y}\ b^T y,\ {\rm{s.t.}} \ L_{11}(y)\succeq 0,  \ L_{12}(y) =0, \ L_{22}(y) =0
	\]
	and the optimal value of {\bf RD} is equal to $v({\bf D})$. 
	\item Under the setting of item 1, 
	let 
	\begin{align*}
		{\cal T}&\coloneqq\left\{ 
		tI=t \begin{pmatrix}  I_{11} & 0 \\ 0 &  I_{22}\end{pmatrix}
		 \in \S^n \,\middle|\, t\in \Re\right\},\\ 
	{\cal L} &\coloneqq \left\{
	\left(\begin{array}{cc} 0  & L_{12}(\hat y) \\ L_{12}(\hat y)^T & L_{22}(\hat y) \end{array}\right)\in \S^n \,\middle|\, \hat y\in \Re^m\right\},
	\end{align*}
	where $I_{11}\in \S^r$ and $I_{22}\in \S^{n-r}$ are identity matrices.  
	Note that ${\cal L}$ is a linear space because it contains the zero matrix (see {\bf RD}).
	We take ${\cal T}+{\cal L}={\cal T}\oplus{\cal L}$ as perturbation space, and 
	consider the following SDP 
	\[
	{\bf RD}(S)\ \ \ 
	\max_{y}\ b^T y,\ {\rm{s.t.}} \ L_{11}(y)+s_{11} I _{11} \succeq 0,  \ L_{12}(y) =S_{12}, \ L_{22}(y)+s_{11}I_{22} =S_{22}
	\]
	obtained by adding the perturbation 
	\[
	S=
	\begin{pmatrix}
	S_{11} & S_{12} \\ S_{12}^T & S_{22}
	\end{pmatrix}
	=s_{11} I + \begin{pmatrix}
	0 & S_{12} \\ S_{12}^T & S_{22}-s_{11}I_{22}
	\end{pmatrix}
      \in  {\cal T}\oplus {\cal L}
	\] 
	to {\bf RD},
	where $S_{11} \in \S^r$, $S_{22} \in \S^{n-r}$ and $S_{12} \in \Re^{r\times (n-r)}$, and 
	$s_{11}\in \Re $ is the $(1,1)$-element of $S$.
	We note that $S$ is decomposed
	as the sum of $s_{11}I\in {\cal T}$ and $S-s_{11}I\in {\cal L}$, which implies that $L_{22}(\hat y)=S_{22}-s_{11}I_{22}$ holds for some $\hat y\in \Re^m$.
	
	\ \ 
	Let $w(S)$ be the optimal value function of the perturbed system\footnote{Compared with {\bf RD}, the first constraint of ${\bf RD}(S)$ is perturbed by $S_{11}=s_{11}I_{11}$, the second is shifted by $S_{12}$ from $0$, and the third is shifted by $S_{22}-s_{11}I_{22}$ from $0$.}  ${\bf RD}(S)$.  
	Then, {$w(0)=v({\bf D})$} and $w(S)$ is continuous at $S=0$.  
%
	
	\item Under the setting of item~1, if the singularity degree of \ref{eq:d} is one, then there exists 
	a nonzero $X \in \PSDcone{n}$ of rank $n-r$ satisfying
		\begin{equation}\label{bdness}
	C \bullet X = 0,\ \   
	A^i \bullet X  = 0,\ i=0,1, \ldots,m, \ \ X\succeq 0
	\end{equation}
	with the form
	\begin{equation}\label{fr_dir}
	X = \left(\begin{array}{cc} 0 & 0 \\ 0 &  X_{22}\end{array}\right), \ \ \ X_{22} \succ 0.
	\end{equation}
	
\end{enumerate}
	
\end{proposition}
\begin{proof}
Items 1 and 3 are an amalgam of well-known results about $\PSDcone{n}$ and minimal faces, so we only present a sketch of the proof.
Let $\stdAff \coloneqq C + \mathcal{L}$, where $\mathcal{L}$ is the span of the $A^i$ in \ref{eq:d}.
Let $\stdFace$ the minimal face of $\PSDcone{n}$ containing 
$\stdAff \cap \PSDcone{n}$. As discussed previously, 
such a face  must be as in \eqref{eq:faces} and have the property
that $(\stdAff \cap \PSDcone{n}) \subseteq  \stdFace$ and $\stdAff \cap (\reInt \stdFace) \neq \emptyset$. Rescaling the $C$ and $A^i$ using $V$ leads to the proof of item~$1$.

%
{
}

As for item~$3$, if the singularity degree is $1$, 
then $\stdFace = \stdCone \cap \{X\}^\perp$, where $X\in \PSDcone{n}\cap \stdAff^\perp$ is nonzero and satisfies  $\stdAff \cap (\reInt (\PSDcone{n} \cap \{X\}^\perp))$. In  view of \eqref{eq:faces},  $X$ has rank $n-r$ and satisfies 
\eqref{bdness} and \eqref{fr_dir}.

It remains to prove item~2. Since $w(0)$ corresponds to the unperturbed problem, we have $w(0) = v(\ref{eq:d})$.
Then, since \textbf{RD} satisfies Slater's condition, continuity of $w(S)$ at $S=0$ in item 2 is obtained as 
a consequence of 
Theorem 4.1.9 of \cite{sdp_handbook}.  For the sake of completeness, we describe a detailed proof in Appendix~\ref{app:proof}.%
\end{proof}

\medskip
\noindent
We note that
even though Proposition~\ref{prop:sd1} is presented for problems in dual format,  a completely analogous discussion 
can be done for \ref{eq:p} if we let $\stdCone \coloneqq \PSDcone{n}$, $\stdAff \coloneqq \{X \in \S^n \mid A^{i} \bullet X = b_i, i = 1,\ldots,m \}$ in \eqref{eq:feas}.

\section{Main Result and Proof}\label{sec:main}

In this section, we show continuity 
of the limiting pd-regularized optimal value function $\av$ defined in \eqref{eq:av} at $\theta =0$ and $\theta = \pi/2$ under the assumption 
that the singularity degree of {\bf P} and {\bf D} is one. We also show bijectivity of $\av$.
More precisely, we prove the following.
\begin{theorem}\label{theo:main}
	Suppose that {\bf P} and {\bf D} are feasible.
	Then the following statements hold on the limiting pd-regularized optimal value function $\av:[0,\pi/2]\to \Re$.
	\begin{enumerate}
		\item If the singularity degree of {\bf D} is one, 
		then $\av(\theta)$ is continuous at $\theta=\pi/2$.
		\item If the singularity degree of {\bf P} is one,
		then $\av(\theta)$ is continuous at $\theta=0$.
		\item 
		If the singularity degree of both {\bf P} and {\bf D} is one,
		then $\av(\theta)$ is continuous at $\theta=0$ and $\theta=\pi/2$.
		Furthermore, $\av$ is a monotonically decreasing continuous bijective function from $[0,\pi/2]$ to $[v({\bf D}),v({\bf P})]$.
	\end{enumerate}
\end{theorem}

In \cite{TLMO2019}, we introduced the function
$\tilde v: \Re_+ \cup \{\infty\} \to \Re$ as 
\begin{equation}\label{aaa1}
\tilde v(\beta) \coloneqq \lim_{t\downarrow 0} v(t, t\beta)\ (0 \leq \beta < \infty),\ \tilde v(\infty)\coloneqq\lim_{t\downarrow 0} v(0,t).
\end{equation}
In what follows, we define the function $\bar v: \Re_+ \cup \{\infty\} \to \Re$ given by
\begin{equation}\label{bbb1}
\bar v(\alpha) \coloneqq\lim_{t\downarrow 0} v(t\alpha,t)\ (0 \leq \alpha < \infty),\ \ \bar v(\infty)\coloneqq\lim_{t\downarrow 0} v(t,0).
\end{equation}
The existence of the limit in \eqref{aaa1} and \eqref{bbb1} is shown  in \cite[Theorem 1]{TLMO2019}.
We see, under the convention that $1/0 = \infty$ and $1/\infty=0$, that
\begin{equation}\label{eq:bartilde}
\bar v(\alpha) = \tilde v \left(\frac{1}{\alpha}\right),\ \ \  \ \tilde v(\beta) = \bar v \left(\frac{1}{\beta}\right). 
\end{equation}
It also follows that
\begin{equation}\label{eq:av_bar_v}
\av(\theta)=\lim_{t\downarrow 0} v(t\cos\theta,t\sin\theta) =\tilde v(\tan\theta)=\bar v(\cot\theta).
\end{equation}
We provide the following proposition on a few fundamental
properties of $\tilde v$ and $\bar v$. 

\begin{proposition}\label{barv}
	The following items hold.
\begin{enumerate}
\item 
$\tilde v(\beta)$ is monotone decreasing on $\Re_+\cup\{\infty\}$, convex on $\Re_+$ and 
continuous on $\Re_{+}\backslash\{0\}$.
\item 
$\bar v(\alpha)$ is monotone increasing on $\Re_+\cup\{\infty\}$, concave on $\Re_+$ and 
continuous on $\Re_{+}\backslash\{0\}$.
\end{enumerate}
\end{proposition}

\begin{proof}
Monotonicity and convexity of $\tilde v$ is given in  \cite[Theorem 4]{TLMO2019}.
Given the convexity of $\tilde v$ on $\Re_+$,  
continuity of $\tilde v(\alpha)$ on $\Re_+\backslash\{0\}$ follows from the well-known fact 
that  a convex function is continuous over the 
relative interior of its domain, e.g., \cite[Theorem~10.1]{rockafellar}.
Item~2 is the dual counterpart to item 1, and follows in a similar manner as outlined below.
It was shown in  \cite[item~2 of Proposition 2]{TLMO2019} that for any fixed $t>0$,
$v(t\alpha, t)$ is a concave function in $\alpha$. Furthermore, since the feasible region of ${\bf D}(t\alpha, t)$
gets larger as $\alpha$ increases (with respect to the $y$ variable) while the objective function is unchanged, $v(t\alpha, t)$ is
a monotone increasing function in $\alpha$ for any fixed $t>0$.
Thus, $v(t\alpha, t)$ is monotone increasing and concave in $\alpha$ for any fixed $t>0$.
Since $\bar v(\alpha)=\lim_{t\rightarrow 0} v(t\alpha, t)$, 
monotonicity and concavity of $\bar v$ is proved by taking the limit as $t\rightarrow 0$
in the similar manner as was done in the proof of monotonicity and convexity of $\tilde v$ in 
\cite[Theorem 4]{TLMO2019}.
\end{proof}


\medskip

Now we are ready to prove Theorem~\ref{theo:main}.   
Before proceeding to the proof, we comment on
the linear independence of the $A^1, \ldots, A^m$.  Theorem~\ref{theo:main} itself does not require the assumption of linear independence  of $A^1, \ldots, A^m$ to hold.  
Nevertheless, we assume linear independence of $A^1, \ldots, A^m$ in some parts of the 
proof.  We remark that this is not an essential assumption and it  avoids certain unnecessary complications.
Indeed, even if $A^1, \ldots, A^m$ are not linearly independent, we can choose a subset $\{A^{i_1}, \ldots, A^{i_p}\}$, say,
as a basis of the linear space spanned by $A^1, \ldots, A^m$.  The pd-regularized
optimal value function $v(\varepsilon, \eta)$ remains unchanged no matter whether $\{A^1, \ldots, A^m\}$ or
$\{A^{i_1}, \ldots, A^{i_p}\}$ is used for representing the SDP under consideration.


\subsection{Proof of item~1.}
First we prove item~1 of Theorem~\ref{theo:main}.
We start by recalling our main assumption for this proof.
\medskip

{\em We assume feasibility of {\bf P} and
	{\bf D}, and assume that the singularity degree of {\bf D} is one.}

\medskip
Furthermore, we assume that the problem data has been rescaled as in item~1 of Proposition~\ref{prop:sd1}. And, since the singularity degree is assumed to be one, we can further assume that we are in the setting of item~3 of Proposition~\ref{prop:sd1}.
With that, by definition, $v(t\alpha, t)$ is the optimal value of the following problem,
\begin{equation}\label{eq:dtat}
{\bf D}(t\alpha, t) \ \ \ 
\max\ b^T y -t(l(y)-c)\ {{\rm{s.t.}}}\ L(y)+t\alpha I \succeq 0,
\end{equation}
where we define 
$$l(y)\coloneqq L(y)\bullet I,\ c\coloneqq C\bullet I.$$
In what follows we will use the following convention, given an arbitrary $Z \in \S^n$, we will use $Z_{11}, Z_{22}$ and $Z_{12}$ to denote the blocks of $Z$  according to the block division in Proposition~\ref{prop:sd1}, so that $Z = \begin{pmatrix}
Z_{11} & Z_{12} \\ Z_{12}^T & Z_{22}
\end{pmatrix}$ with $Z_{11} \in \S^r, Z_{22} \in \S^{n-r}, Z_{12} \in \Re^{r\times (n-r)}$.

We start the proof with the following two preliminary lemmas.
\begin{lemma} \label{l3} \ 
	The following statements hold:
	\begin{enumerate}
		\item  $v({\bf D}) \leq \bar v(\alpha) \leq v({\bf P})$.
		\item For $t >0$, $v(t,0)$ is finite and $\lim_{t\downarrow 0}v(t,0)=\bar v(\infty)  = v({\bf P})$.
		\item For $t >0$, $v(0,t)$ is finite and $\lim_{t\downarrow 0}v(0,t)=\bar v(0)  = v({\bf D})$.
	\end{enumerate}
\end{lemma}

\begin{proof}
	Item 1 readily follows from Theorem~\ref{thm:tlmo}, since $\bar v$ is the function $\av$ with a different parametrization, see \eqref{eq:av_bar_v}.
	To prove item 2, we observe that {\bf P} and ${\bf P}(t,0)$ have the same feasible region.
	Since {\bf P} is feasible, ${\bf P}(t,0)$ is feasible and $v(t,0)<+\infty$.
	Since ${\bf D}$ is feasible, ${\bf D}(t,0)$ is also feasible.  This means that $-\infty < v(t,0)$.
	Thus, $v(t,0)$ is finite.  It follows from item 1 of Theorem~\ref{thm:tlmo} that  $\lim_{t\downarrow0} v(t,0)  = v({\bf P})$.
	Item 3 is the dual counterpart of item 2 and follows analogously. 
\end{proof}

{
	\begin{lemma} \label{bound} 
		There exists a constant $M >0 $ such that for every $t \geq 0 $ the following implication holds:
		\begin{equation} \label{bound1}
		L_{22}(y) + tI_{22} \succeq 0\quad  \Longrightarrow\quad tMI_{22}  \succeq  L_{22}(y) + tI_{22} \succeq 0.
		\end{equation}
	\end{lemma}
}

\begin{proof}
Recall that we are under the setting of item~1 of Proposition~\ref{prop:sd1} and \ref{eq:d} has singularity degree one.
With $X$ as in item~3 of Proposition~\ref{prop:sd1}, the rank of $X_{22}$ is $n-r$. Since $X_{22}$ is positive definite, 
there exists a constant $\kappa > 0$ such that
\begin{equation}\label{eq:bound}
\norm{Y} \leq \kappa X_{22}\bullet Y, \qquad \forall  Y \in \PSDcone{n-r},
\end{equation}
e.g., \cite[Lemma~26]{Lourenco17}.

Next, given some arbitrary $t > 0$, we define
\[
\mathcal{C}_t  \coloneqq  \{Y\in \PSDcone{n-r} \mid Y=L_{22}(y)+ tI_{22} \succeq 0 ,\   y \in \Re^m \}.
\] 

By the assumptions on $X$ we have $0 = X \bullet A^i = A^i_{22} \bullet X_{22}$ and $0 = C \bullet X = C_{22} \bullet X_{22}$.
That is, $L_{22}(y)\bullet X_{22} = 0$ holds for every $y$. 
In view of \eqref{eq:bound}, for $Y \in \mathcal{C}_t$, since $Y$ belongs to $\PSDcone{n-r}$ as well, we conclude that the following bound holds.
\[
\norm{Y} \leq \kappa X_{22} \bullet (L_{22}(y) + tI_{22}) = \kappa t (X_{22} \bullet I_{22}).
\]

In particular, letting $M \coloneqq \kappa(X_{22} \bullet I_{22}),$
we conclude that the maximum eigenvalue of an arbitrary $Y \in \mathcal{C}_t$ 
satisfies $\lambda_{\max}(Y) \leq tM$ and $M$ does not depend on $t$. In particular,  
$tMI_{22} \succeq L_{22}(y) + tI_{22}$ holds  for all $y$ such that $	L_{22}(y) + tI_{22} \succeq 0 $.
\end{proof}

\medskip
In the following, we define 
\begin{equation}\label{eq:ls}
\begin{aligned}
	l_{12}^2(y)& \coloneqq L_{12}(y)\bullet L_{12}^T(y),\\
	l_{11}(y) & \coloneqq L_{11}(y)\bullet I_{11}, \\
	l_{22}(y) & \coloneqq  L_{22}(y)\bullet I_{22}.
\end{aligned}
\end{equation}
We will use $M$ as a global constant satisfying  \eqref{bound1}
in Lemma \ref{bound}.
Let us consider the following problem.
\begin{equation*}
{\bf RD1}(\alpha,t)\ \ \ 
\max\ b^T y - \frac{l_{12}^2(y)}{M \alpha},\ \ {\rm{s.t.}}\ \left(\begin{array}{cc} L_{11}(y)+t\alpha I_{11} & L_{12}(y)\\ L_{12}^T(y) & L_{22}(y) +t\alpha I_{22} \end{array}\right)\succeq 0.
\end{equation*}
Let $u_1(\alpha,t)$ be the optimal value function of ${\bf RD1}(\alpha,t)$.  
We  show that $v(t\alpha,t)$ is majorized by the optimal value $u_1(\alpha,t)$
as follows.
\begin{lemma}\label{6}
	For $t > 0$ and $\alpha > 0$, we have 
	\begin{equation}\label{l6}
	v(0,t)\leq 
	v(t\alpha, t) \leq u_1(\alpha,t)+t^2\alpha n  + t c.
	\end{equation}
\end{lemma}
\begin{proof}
	We prove the first inequality. Recall that $v(0,t)$ is the optimal value of  ${\bf D}(0,t)$ and  $v(t\alpha,t)$ is the optimal value of  ${\bf D}(t\alpha,t)$.
	Since the objective functions of the two problems are identical and the feasible region of ${\bf D}(t\alpha,t)$ contains the feasible region of ${\bf D}(0,t)$ (with respect to the $y$ variable),
	we have $v(0,t)\leq v(t\alpha,t)$.
	
	Now we prove the second inequality. Recalling the definitions in \eqref{eq:ls} and \eqref{eq:dtat}, 	${\bf D}(t\alpha, t)$ can be written as follows. 
		\begin{equation}\label{dalpha}
		\max\ b^T y + t(c- l_{11}(y)-l_{22}(y)),\ {\rm{s.t.}}\left(\begin{array}{cc} L_{11}(y)+t\alpha I_{11} & L_{12}(y)\\ L_{12}^T(y) & L_{22}(y)+t\alpha I_{22} \end{array}\right)\succeq 0.
		\end{equation}
		We note that ${\bf D}(t\alpha, t)$ and
		${\bf RD1}(\alpha,t)$ share the same feasible region. Furthermore, 
		${\bf D}(t\alpha, t)$ satisfies Slater's condition so there exists a sequence of $y^k$ corresponding to feasible solutions to  ${\bf D}(t\alpha, t)$ such that 
		the corresponding matrices are all positive definite and $(b^T y^k -tl_{11}(y^k)-tl_{22}(y^k)+tc)$ converges to the optimal value ${\bf D}(t\alpha, t)$. Thus, in order to establish the second inequality in \eqref{l6}, it is enough to examine the $y$'s that correspond to positive definite matrices.
		
		So, let $y$ be feasible solution to ${\bf D}(t\alpha, t)$ associated to a positive definite matrix.
		First, we find an upper bound on  $- tl_{11}(y)$.	Computing the Schur complement, we obtain
	\begin{equation}\label{eq:lem_aux}
	L_{11}(y)+t\alpha I_{11}- L_{12}(y) (L_{22}(y)+t\alpha I_{22})^{-1} L_{12}^T(y) \succeq 0.
	\end{equation}
	Since $L_{22}(y)+t\alpha I_{22} \succeq 0$, we obtain $Mt\alpha I_{22} \succeq (L_{22}(y)+t\alpha I_{22})$ from \eqref{bound1}. So each eigenvalue of 
	$(L_{22}(y)+t\alpha I_{22})$ is less than or equal to $Mt\alpha$.
	This implies that $(L_{22}(y)+t\alpha I_{22})^{-1} \succeq \frac{I_{22}}{Mt\alpha}$. Therefore, we have
	\[
	L_{12}(y)\left((L_{22}+t\alpha I_{22})^{-1} - \frac{I_{22}}{Mt\alpha}\right)L_{12}(y)^T \succeq 0,
	\]
	which implies that
	\[
		L_{12}(y)(L_{22}+t\alpha I_{22})^{-1}L_{12}(y)^T \succeq   L_{12}(y)\frac{I_{22}}{Mt\alpha}L_{12}(y)^T.
	\]
Then, in view of \eqref{eq:lem_aux}, we obtain
	\[
	L_{11}(y)+t\alpha I_{11}- \frac1{tM\alpha} L_{12}(y)  L_{12}(y)^T \succeq 0.
	\]
	Taking the inner-product with $I_{11}$ and multiplying by $t$, we obtain
	\begin{equation}\label{tl11}
	t l_{11}(y) + t^2\alpha r -\frac1{M\alpha} l_{12}^2(y) \geq 0.
	\end{equation}
	Since  $L_{22}(y)+t\alpha I_{22}\succeq 0$, we have  $-tl_{22}(y) \leq t^2\alpha (n-r)$.
	From this inequality and \eqref{tl11}, we see
	that the objective function of \eqref{dalpha} is majorized by the objective function of {\bf RD1} plus
	a constant $t^2\alpha n+tc$ as follows. 
	\[
	 b^T y + t(c- l_{11}(y)-l_{22}(y)) \leq b^T y -\frac{l_{12}^2(y)}{M \alpha}+t^2\alpha n+ tc.
	\]
	Since the optimal value of \eqref{dalpha} (or equivalently ${\bf D}(t\alpha, t)$) is $v(t\alpha,t)$ and the optimal value of 
	${\bf RD1}(\alpha,t)$ is $u_1(\alpha,t)$,
	the second inequality of \eqref{l6} follows immediately from this inequality.
\end{proof}

\begin{lemma}\label{lem:rd1}
	Let $\alpha > 0$. Then, there exists 
	a constant $\hat t_\alpha>0$ depending on $\alpha$ such that,\ if $t \in (0,\hat t_\alpha)$ and $\{y^k\}$ is 
	an optimal sequence of ${\bf RD1}(\alpha,t)$\footnote{${\bf RD1}(\alpha,t)$ is not a linear SDP, but we can define optimal sequences analogously as in Section~\ref{sec:pre}.} then the 
	following bound holds for sufficiently large $k$:
	\begin{equation} \label{zz}
		l_{12}^2(y^k) \leq M(v({\bf P}) - v({\bf D})+2) \alpha =K\alpha,
	\end{equation}
	where $K\coloneqq M(v({\bf P}) - v({\bf D})+2) $ and $M$ is a constant as in Lemma~\ref{bound}. 
	
\end{lemma}

\begin{proof}
	Recall that $\lim_{t\downarrow 0} v(t\alpha, 0) = v({\bf P})$ and 
	let $\hat t_\alpha>0$ be small enough such that  
	\begin{equation} \label{ddda}
		v(t\alpha,0) \leq v({\bf P})+\frac12
	\end{equation}
	holds for $0<t<\hat t_\alpha$.
	
	Let $t \in (0, \hat t_\alpha)$ and let
	$\{y^k\}$ be an optimal sequence to ${\bf RD1}(\alpha,t)$. Each $y^k$ is also a feasible solution to ${\bf D}(t\alpha, 0)$, so the following inequality holds for every $k$
	\begin{equation}\label{ccca}
	b^T y^k \leq v(t\alpha,0).
	\end{equation}
	Next, let $\{y_D^k\}$ be an optimal sequence to $\bf D$. Then $\lim_{k\rightarrow\infty} b^T y_D^k
	=v({\bf D})$. 
	Recalling that we are under the setting of Proposition~\ref{prop:sd1}, we have that 
	$L_{12}(y_D^k) = 0$ holds which implies 
	that $l_{12}^2(y_D^k) = 0$ holds as well, see \eqref{eq:ls}. 
	Furthermore, each $y_D^k$ is feasible to ${\bf RD1}(\alpha,t)$ as well. 
	In view for these facts, for sufficiently large $k$, 
	we have  $v({\bf D})-1   \leq b^T y_D^k \leq v({\bf D})$ which leads to
	\begin{align*}
		v({\bf D})-1 &\leq b^T y_D^k \\ 
		&= b^T y_D^k - \frac1{M\alpha}l_{12}^2(y_D^k)\\
		& \leq v({\bf RD1}(\alpha,t))\\
		& \leq b^T y^k - \frac1{M\alpha} l_{12}^2(y^k) +\frac12 \\
		&\leq v(t\alpha, 0)- \frac1{M\alpha} l_{12}^2(y^k)+\frac12 \\
		& \leq v({\bf P})+1- \frac1{M\alpha} l_{12}^2(y^k),
	\end{align*}
	where the first equality holds because $l_{12}^2(y_D^k)=0$, the second inequality holds because $y_D^k$ is feasible for
	${\bf RD1}(\alpha,t)$, and the third inequality holds (for sufficiently large $k$) because $\{y^k\}$ is an
	optimal sequence.
	Moreover, the last two inequalities follow from \eqref{ccca} and \eqref{ddda}, respectively.
	Overall, we have
	\[
	\frac1{M\alpha} l_{12}^2(y^k) \leq v({\bf P}) - v({\bf D})+2,
	\]
	for sufficiently large $k$.
\end{proof}

Now we consider
\[
{\bf RD2}(\alpha,t)\ \ \ 
\max\ b^T y ,\ \ {\rm{s.t.}}\ \left(\begin{array}{cc} L_{11}(y)+t\alpha I_{11} & L_{12}(y)\\ L_{12}^T(y) & L_{22}(y) +t\alpha I_{22} \end{array}\right)\succeq 0,\ \ 
l_{12}^2(y) \leq K\alpha,
\]
where $K$  is as in Lemma~\ref{lem:rd1}.
Let us denote by $u_2(\alpha,t)$ the optimal value of ${\bf RD2}(\alpha,t)$.
\begin{lemma}\label{lem:bar_u}
For $\alpha > 0$, the function defined by	$\bar u(\alpha) \coloneqq \lim_{t\downarrow 0} u_2(\alpha,t)$ is well-defined and finite.
\end{lemma}
\begin{proof}
	We note that $v(t\alpha,0)$ is the optimal value of ${\bf D}(t\alpha, 0)$ and ${\bf D}(t\alpha, 0)$ is precisely the problem obtained by removing the constraint ``$l_{12}^2(y) \leq K\alpha$'' from {\bf RD2}. Therefore,  we have \[u_2(\alpha,t)\leq v(t\alpha,0).\]
	By item 2 of Lemma \ref{l3}, for sufficiently small $t > 0$, $v(t\alpha,0)$ is bounded from above by a finite value, which implies that the same is true for $u_2(\alpha,t)$.  
	For fixed $\alpha$, if we increase $t$, the feasible region of ${\bf RD2}(\alpha,t)$ enlarges (more precisely, it does not shrink). 
	Therefore, $u_2(\alpha,t) \leq u_2(\alpha,\hat t)$ holds if $ 0 < t  \leq \hat t$, namely,
	$u_2(\alpha,t)$ is a monotone increasing function with respect to $t$. 
	In addition,
	$u_2(\alpha,t)$ 
	is  bounded below by $v({\bf D})$, since any feasible solution to {\bf D} is feasible to ${\bf RD2}(\alpha,t)$ (again, we recall that we assume that Proposition~\ref{prop:sd1} holds, so $L_{12}(y) = 0$ and $L_{22}(y) = 0$ if $y$ is feasible to {\bf D}).
	Therefore,  $\lim_{t\rightarrow 0} u_2(\alpha,t)$ exists and is finite, thus showing that $\bar u(\alpha)$ is a well-defined function assuming finite values.
\end{proof}
\medskip

Now we are ready to finish the proof of item~1 of Theorem~\ref{theo:main}.
For $\alpha > 0$ and $\hat t_{\alpha} > 0$ as in Lemma~\ref{lem:rd1}, an optimal sequence $\{y^k\}$ to ${\bf RD1}(\alpha,t)$ satisfies \eqref{zz} for sufficiently large $k$. Furthermore, for each $y^k$, the objective value associated to ${\bf RD2}(\alpha,t)$ is greater or equal than the corresponding objective value associated to the problem   ${\bf RD1}(\alpha,t)$. In view of these facts and of Lemma~\ref{lem:rd1}, we see that $u_1(\alpha,t) \leq u_2(\alpha,t)$. Combined with Lemma~\ref{6}, we obtain{
	\begin{equation}\label{eq:v0t}
	v(0,t) \leq v(t\alpha,t) \leq u_1(\alpha,t)+t^2\alpha n  + t c \leq u_2(\alpha,t)+t^2\alpha n  + t c.
	\end{equation}
Recalling the definitions of $\bar v$ and $\bar u$ in 
\eqref{bbb1} and Lemma~\ref{lem:bar_u}, respectively, and invoking item~1 of Theorem~\ref{thm:tlmo}, we let $t\rightarrow 0$ in \eqref{eq:v0t} in order to obtain 
\[
v({\bf D}) \leq \bar v(\alpha)\leq \bar u(\alpha).
\]
Therefore, if we can show that 
\begin{equation}\label{limu2}
\lim_{\alpha\downarrow 0} \bar u(\alpha) =
\lim_{\alpha\downarrow 0}\lim_{t\downarrow0} u_2(\alpha,t) = v({\bf D}),
\end{equation}
then $\lim_{\alpha \downarrow 0} \bar v(\alpha)  = v({\bf D})$ will hold as well, which, in view of \eqref{eq:av_bar_v} and Theorem~\ref{thm:tlmo} implies the continuity of $v_a$ at $\theta = \pi/2$, which is what we wanted to show.  
We will prove \eqref{limu2} by contradiction.



\medskip\noindent
{\bf {\it Proof of \eqref{limu2}}:}
	First, we recall that, in view of item 1 of Proposition~\ref{prop:sd1}, every $y$ that is feasible to {\bf D} is feasible to ${\bf RD2}(\alpha,t)$, for every $\alpha > 0$, $t > 0$. Therefore
	\[
	v({\bf D}) \leq u_2(\alpha,t), \qquad \forall \alpha > 0, \forall t > 0,
	\]	
	which implies 
	that
	\[
	v({\bf D}) \leq \bar{u}(\alpha), \qquad \forall \alpha > 0.	
	\]
	If \eqref{limu2} does not hold, there exists $\delta > 0$ and 
	a positive sequence $\{\alpha^k\}$ such that $\lim_{k\rightarrow \infty}\alpha^k=0$
	for which 
	\[0 < 2\delta \leq |\bar u(\alpha^k) - v({\bf D})|  = \bar u(\alpha^k) - v({\bf D})\] 
	holds for sufficiently large $k$.
	Next, since $\bar u(\alpha^k) = \lim_{t\downarrow0} u_2(\alpha^k,t)$, for any given $k$, we can 
	pick $t^k\in (0,1/k] $ such that  $|\bar u(\alpha^k) -  u_2(\alpha^k,t^k)| \leq 1/k$.

	Therefore, there exists a
	sequence $(\alpha^k,t^k )\rightarrow (0, 0)$
	such that 
	\[
	u_2(\alpha^k,t^k)> v({\bf D})+\delta
	\] 
	holds for sufficiently large $k$.
	Then, for each $k$ sufficiently large, there exists a feasible solution $y^k$ to 
	${\bf RD2}(\alpha^k,t^k)$ satisfying
	\[
		b^T y^{k} \geq v({\bf D})+\delta.
	\]
	Next, we define
$S_{11}^k \coloneqq t^k\alpha^k I_{11}$,	$S_{22}^{k} \coloneqq L_{22}(y^{k})+t^k\alpha^k I_{22}$,  $S_{12}^k \coloneqq L_{12}(y^{k})$ and 
$S^k \coloneqq \begin{pmatrix} S_{11}^k & S_{12}^k \\ (S_{12}^k)^T & S_{22}^k\end{pmatrix}$.
Note that $S^k \in {\cal T}\oplus {\cal L}$ holds, where ${\cal T}\oplus {\cal L}$ was introduced as perturbation space in item~2 of Proposition~\ref{prop:sd1}.  
By definition, $y^{k}$ is a feasible solution to
the following problem.
\begin{eqnarray*}
	{\bf RD3}&\ &\ \ 
	\max_{y}\ b^T y ,\ \ {\rm{s.t.}}\ \left(\begin{array}{cc} L_{11}(y)+S_{11}^k  & L_{12}(y)\\ L_{12}^T(y) & L_{22}(y) +t^k\alpha^k I_{22} \end{array}\right)\succeq 0,  \\
	&& \ \ \ \ \ \ \ \ \ \ \ \ \ \ \ \ \ \ \ \ \ \ \ L_{22}(y)+t^k\alpha^k I_{22} =S_{22}^k, \ \  L_{12}(y)=S_{12}^k,\ \ l_{12}^2(y) \leq K\alpha^k.
\end{eqnarray*}
Because $y^k$ is also feasible to ${\bf RD2}(\alpha^k,t^k)$, 
we have $S_{22}^k=L_{22}(y^{k})+t^k\alpha^k I_{22} \succeq 0$ and by Lemma~\ref{bound}, $S_{22}^k$ goes to $0$ as $t^k$ and $\alpha^k$ goes to zero.
Similarly, since $L_{12}(y^k) \bullet L_{12}(y^k) =  l_{12}^2(y^k) \leq K\alpha^k$ holds, 
$S_{12}^k$ also goes to zero as $\alpha^k$ goes to $0$.
By definition, $S_{11}^k$ goes to zero as $t^k$ and $\alpha^k$ goes to $0$. So, overall $S^k$ goes to $0$ as $k \to \infty$.

Finally, recalling item~2 of Proposition~\ref{prop:sd1}, we note that $y^k$ is also a feasible solution to 
${\bf RD}(S^k)$ satisfying $b^T y^k \geq v({\bf D})+ \delta$.
Therefore, $w(S^k) \geq v({\bf D})+ \delta$ and $S^k\rightarrow 0$ holds as $k$ goes to infinity.
But this contradicts the continuity of $w(S)$ at $S=0$, that is, the conclusion of item 2 of Proposition~\ref{prop:sd1}.
\hfill$\qed$\\ 
Thus, \eqref{limu2} must hold, and the proof of item~1 of Theorem~\ref{theo:main} is complete.

\subsection{Proof of item~2}
Now, we proceed to prove item 2 of Theorem \ref{theo:main}. Item 2 is the dual counterpart of item 1, and therefore
one may argue that item 2 follows automatically from item~1 by primal-dual symmetry.  But to be more precise, we proceed as follows.  
Our purpose is to show that $\av(\theta)$ is continuous at $\theta =0$, which is
equivalent to continuity of $\tilde v(\beta)$ at $\beta=0$.  To this end, we rewrite {\bf P} in 
the dual format  and apply item 1.

Without loss of generality, we may assume the matrices $A_1,\ldots, A_m$ are linearly independent.
Let $ \bar n :=n\times(n+1)/2-m$.
We rewrite ${\bf P}(\varepsilon,\eta)$ in the dual format by taking a basis $A_{\perp}^1, \ldots, A_\perp^{\bar n}$ of the space ${\cal X}:=\{X\in \S^{n}\mid  A^i\bullet X = 0,\ i=1, \ldots, m\}$.  Let  $X^*$ be a $n\times n$ matrix satisfying $A^i\bullet X_i^* = b_i$ for $i=1, \ldots, m$.   Representing ${\cal X}$ as $\{\widetilde X\in \S^n \mid \widetilde X=-\sum_{i=1}^{\bar n}A_\perp^i\tilde y_i \}$ and reversing the sign of the objective function to flip ``max'' and ``min'', we obtain the following 
representation of  ${\bf P}(\varepsilon,\eta)$.

\begin{equation*}
\begin{array}{ll}
\ & \max_{\tilde y,\widetilde X} \sum_{i=1}^{\bar n} (\tilde b_i + \varepsilon A_\perp^i\bullet I)\tilde y_i 
-(C+\varepsilon I)\bullet (X^* + \eta I) \\
\ & \\
\ &{\rm{s.t.}} \ (X^* + \eta I) - \sum_{i=1}^{\bar n} A_\perp^i \tilde y_i  = \widetilde X,\ \ \ \widetilde X\succeq 0,
\end{array}
\end{equation*}
where $\tilde b_i = C^i\bullet A_\perp^i$ for $i=1, \ldots, \bar n$.  The optimal value of this problem is
$-v(\varepsilon, \eta)$ if $(\varepsilon, \eta)\not=0$, and $-v({\bf P})$ if $(\varepsilon, \eta)=0$.

Dropping the constant term $-(C+\varepsilon I)\bullet (X^* + \eta I) $,  we obtain the following problem. 
\begin{equation*}
\begin{array}{ll}
{\bf D}_{\rm P}(\eta,\varepsilon):& \max_{\tilde y,\widetilde X} \sum_{i=1}^{\bar n} (\tilde b_i + \varepsilon A_\perp^i\bullet I)\tilde y_i \\
\ & \\
\ &{\rm{s.t.}} \ (X^* + \eta I) - \sum_{i=1}^{\bar n} A_\perp^i \tilde y_i  = \widetilde X,\ \ \ \widetilde X\succeq 0.
\end{array}
\end{equation*}
The dual problem of ${\bf D}_{\rm P}(\eta,\varepsilon)$  is
\begin{equation*}
\begin{array}{ll}
{\bf P}_{\rm D}(\eta,\varepsilon):&\min_{\widetilde S} \ (X^* + \varepsilon I)\bullet \widetilde S  \\ 
\ &{\rm{s.t.}}\ A_\perp^i \bullet \widetilde S = \tilde b_i+ \eta A_\perp^i\bullet I,\ i=1, \ldots, {\bar n},\ \widetilde S \succeq 0,
\end{array}
\end{equation*}
Note that ${\bf P}_{\rm D}(\eta,\varepsilon)$ coincides with ${\bf D}(\varepsilon,\eta)$ represented in standard form using $A_\perp^i$ ($i=1,\ldots,\bar n$) and $X^*$, except that $\min$ and $\max$ are flipped by 
multiplying the objective function by $-1$ and the constant term  $-(C+\varepsilon I)\bullet (X^* + \eta I)$
is dropped.   

Let us denote by $v_1(\eta,\varepsilon)$ the common optimal value of  ${\bf P}_{\rm D}(\eta,\varepsilon)$ and
  ${\bf D}_{\rm P}(\eta,\varepsilon)$ (if it exists).
Since $X$ part of ${\bf P}(\varepsilon,\eta)$ and  $\widetilde X$ part of 
${\bf D}_{\rm P}(\eta,\varepsilon)$ coincide, and since
$S$ part of ${\bf D}(\varepsilon,\eta)$ and $\widetilde S$ part of ${\bf P}_{\rm D}(\varepsilon,\eta)$ coincide, $v_1(\eta,\varepsilon)$
is well-defined if and only if $v(\varepsilon,\eta)$ is well-defined, and the two functions are connected by the
relation
\begin{equation}\label{v1}
v_1(\eta,\varepsilon)=-v(\varepsilon,\eta) + (C+\varepsilon I)\bullet (X^* + \eta I)
\end{equation}
for $(\varepsilon,\eta)\not=(0,0)$.  In the case of $(\varepsilon,\eta)=(0,0)$, we have
\begin{equation}\label{diffopt} 
 v({\bf D}_{\rm P}(0,0))=-v({\bf P})+C\bullet X^*,\ \ 
v({\bf P}_{\rm D}(0,0))=-v({\bf D})+C\bullet X^*.
\end{equation}

Now, consider the primal-dual pair ${\bf D}_{\rm P}(0,0)$ and ${\bf P}_{\rm D}(0,0)$.
It follows from \eqref{diffopt} that  ${\bf P}_{\rm D}(0,0)$ and ${\bf D}_{\rm P}(0,0)$ have a finite nonzero duality gap
that is equal to the duality gap between ${\bf P}$ and ${\bf D}$. 
Furthermore, the singularity degree of ${\bf D}_{\rm P}(0,0)$ and ${\bf P}_{\rm D}(0,0)$ is one,
because applying facial reduction to ${\bf D}_{\rm P}(0,0)$ and  ${\bf P}_{\rm D}(0,0)$ corresponds to applying facial reduction to ${\bf P}$ and ${\bf D}$, respectively, which finishes in one step.
Therefore, we can apply the result of item 1 to ${\bf D}_{\rm P}(0,0)$.

We define $\bar v_1(\beta)=\lim_{t\downarrow 0}v_1(t\beta,t)$.
Then it follows from item 1 that $\bar v_1(\beta)$ is continuous at $\beta = 0$.  
Due to \eqref{v1} and the definition of $\bar v_1$, we have
$\bar v_1(\beta) = -\tilde v(\beta)+C\bullet X^*$.  Consequently, 
$\tilde  v(\beta)$ is continuous at $\beta=0$, as we desired.
\endproof

  
\subsection{Proof of item 3}

We move on to a proof of item~3 of  Theorem \ref{theo:main}. 
The first half of item~3 follows immediately from items 1 and 2.  In the following, we deal with the second half of item~3.  Since monotonicity and continuity of $\av$ on $[0,\pi/2]$ was already established, we focus on the bijectivity of $\av$.  To this end, we make use of Proposition~\ref{barv}, which ensures monotonicity, convexity/concavity,
and continuity of $\tilde v$ and $\bar v$. 

In view of \eqref{eq:av_bar_v}, it is enough to argue that 
$\bar{v}$ is bijection from $[0,\infty) \cup \{+\infty\}$ to $[v({\bf D}),v({\bf P})]$ with the convention that $\bar{v}(+\infty) = \lim _{\alpha \to + \infty}\bar{v}(\alpha)= v({\bf P}) $ as in \eqref{bbb1}.

By contradiction, assume that $\bar v(\alpha)$ is (monotone increasing but) not strictly monotone increasing.
Then, there exists an interval $J \subseteq \Re_+$ where $\bar v $ is constant. Slightly  abusing the notation, let us denote this constant by $\vJ$.   There are three cases to be considered:
 \begin{enumerate}[$(i)$]
 	\item $\vJ=\bar v(0) = v({\bf D})$;
 	\item $v({\bf D}) =\bar v(0)<\vJ<\bar v(\infty) = v({\bf P})$;
 	\item $\vJ=\bar v(\infty) = v({\bf P})$.
 \end{enumerate}
 In case $(i)$, we recall that $\bar v$ is monotone increasing, so $v(\bf D)$ is the minimum value of $\bar v$
 throughout $\Re_+$. Since this minimum value is attained in an interval and $\bar v$ is concave, $\bar v$ must be constant throughout $\Re_+$, see \cite[Theorem~32.1]{rockafellar} for the analogous fact about convex functions.
 That is,   $\bar v(\alpha) =  v({\bf D}) < v({\bf P})$ holds for all $\alpha > 0$.  This 
 contradicts the fact that $\lim_{\alpha\rightarrow\infty}\bar v(\alpha)=v({\bf P})$. 
 So case $(i)$ cannot occur.
 
 In case $(ii)$, 
 let $x$ be a point in the relative interior of $J$, 
 then $x$ would be a local maximum of $\bar {v}$ that is not 
 global. This contradicts the concavity of $\bar v$ and cannot occur.
 
 Finally, in case $(iii)$, since $\bar{v}$ is monotone increasing and $v(\bf{P})$ is the maximum value,  $J$ can be taken to be a closed
 half line of the form $[\alpha^*,+\infty)$, with $\alpha^* > 0$.  
 Now we turn our attentions to $\tilde v$ and recall that $\bar v$ and $\tilde v$ are related as in \eqref{eq:bartilde}.  We have $\bar v(\alpha)=\tilde v(1/\alpha)=v({\bf P})$ for all $\alpha \geq \alpha^*$.
 Recall that $\tilde v(0) = v({\bf P})$, and thus  $\tilde v(0)=\tilde v(1/\alpha) = v({\bf P})$, for all  $\alpha \geq \alpha^*$.
 Since $v({\bf P})$ is the maximum value of the convex function $\tilde v$, $\tilde v$ must be constant throughout its domain, again by \cite[Theorem~32.1]{rockafellar}. Therefore, $\tilde v(\beta) = v({\bf P})$ holds for all $\beta \in \Re_+$.
 Therefore, 
 \[
 \lim_{\alpha \downarrow 0} \bar v(\alpha) = \lim_{\alpha \downarrow 0} \tilde v(1/\alpha) =v({\bf P}) > 
 v({\bf D}) = 
 \bar v(0),
 \]
 which contradicts the continuity of $\bar v$ at $\alpha=0$.

\section{A Counter-Example}\label{sec:counter_ex}

In the previous section, we established that the limiting pd-regularized optimal value function $\av$
is a monotonically decreasing continuous bijective function from $[0,\pi/2]$ to $[v({\bf D}), v({\bf P})]$ if 
{\bf P} and {\bf D} admit a finite nonzero duality gap and
the singularity degree of {\bf P} and {\bf D} is one.  In this section, we present a counter-example where
$\av([0,\pi/2])\not=[v({\bf D}), v({\bf P})]$.
The instance is obtained by modifying Example~\ref{ex:gap} and the singularity degree of this instance is two.  We will show that continuity of  $\av(\theta)$ is violated at $\theta=\pi/2$.

Let us consider the following  semidefinite program in dual format:
\begin{equation}\label{modRam}
\max_{y \in \Re^3} \ y_1\ {\rm{s.t.}}\ \ \left(\begin{array}{cccc} 1 -y_1 & 0 & -y_3 & -y_3 \\
0 & -y_2 & -y_1 & 0\\ 
-y_3 & -y_1 & -y_3 & 0\\
-y_3 & 0 & 0 & 0
\end{array}\right) \succeq 0.
\end{equation}
The $(4,4)$-element of the matrix in \eqref{modRam} is zero, this forces $y_3$ to be $0$ in order for $y \in \Re^3$ to be feasible.  This, however, implies that the $(3,3)$-element is $0$, which leads to $y_1 = 0$. Therefore, the optimal value of 
\eqref{modRam} is $0$.

The corresponding primal {\bf P} is
\begin{equation}\label{eq:modRam_P}
\min_{X \in \S^4}\ X_{11}\ \ {\rm{s.t.}}\ X_{11}+ 2 X_{23} = 1, \ X_{22}=0,\ X_{33}+2(X_{13}+X_{14})=0,\ \ X \succeq 0.
\end{equation}
If $X$ is feasible for \eqref{eq:modRam_P}, then 
we must have $X_{22}=0$ and $X_{23}=0$, which implies 
that $X_{11}=1$. Therefore the optimal value of \eqref{eq:modRam_P} is $1$.

In conclusion, there is a duality gap between the primal dual pair of problems \eqref{modRam}, \eqref{eq:modRam_P}.
However, in contrast to Example~\ref{ex:gap}, the singularity degree of \eqref{modRam} is two\footnote{To see that, we observe that the first reducing direction $s_1$ in \eqref{reduce} must be a $4\times 4$ matrix whose only nonzero and positive entry is its (4,4)-element. One can then confirm that the SDP obtained by performing one step of facial reduction with $s_1$ is essentially Example~\ref{ex:gap} which requires one more step of facial reduction to
recover strong feasibility.}, 
so we are outside of the scope of Theorem~\ref{theo:main}.

Now we consider the problem ${\bf D}(\varepsilon,\eta)$ as in \eqref{Dptbd}:
\[
\max_{y \in \Re^3}\ (1+\eta) y_1 + \eta y_2\ + \eta y_3 \ {\rm{s.t.}}\ \left(\begin{array}{cccc} 1+\varepsilon -y_1 & 0 & -y_3 & - y_3 \\
0 & \varepsilon -y_2 & -y_1 & 0\\ 
-y_3 & -y_1 & -y_3+\varepsilon& 0 \\
-y_3 & 0 & 0 & \varepsilon 
\end{array}\right) \succeq 0.
\]

Recalling the definition $\av$ in \eqref{eq:av_bar_v}, we prove the following theorem, which states that the 
range of $\av$ is $\{0, 1\}$. Therefore, in this case, duality gap fails to be filled completely in the sense of Theorem~\ref{theo:main}.

\begin{theorem} The following statements hold for the problem \eqref{modRam}:
\begin{enumerate}
\item $\av(\theta)=v({\bf P})=1\  for \ all \ \theta\in [0,\pi/2)$,
\item $\av(\pi/2)=v({\bf D})=0$.
\end{enumerate} 
\end{theorem}

\begin{proof}
In view of \eqref{eq:av_bar_v}, it is enough to 
analyze $\bar v$ (defined in \eqref{bbb1}).
We will show that 
$\bar v(0)=0$ and $\bar v(\alpha)=1$ for all $\alpha>0$.
Since $v({\bf D})=\bar v(0)=0$ and $v({\bf P})=\bar v(\infty)=1$, this will imply that $\bar v$ is discontinuous at $0$
and is continuous at $+\infty$.
We analyze the problem ${\bf{D}}(t\alpha, t)$:
\begin{equation}\label{aaa}
\max_{y \in \Re^3}\ (1+t) y_1 + t y_2\ + t y_3 \ {\rm{s.t.}}\ \left(\begin{array}{cccc} 1+ t\alpha -y_1 & 0 & -y_3 & - y_3 \\
0 & t\alpha -y_2 & -y_1 & 0\\ 
-y_3 & -y_1 & -y_3+t\alpha& 0 \\
-y_3 & 0 & 0 & t\alpha 
\end{array}\right) \succeq 0.
\end{equation}

We already know that $0 \leq \bar v(\alpha) \leq 1$ for $0 \leq \alpha \leq \infty$. Let $\alpha > 0$ be fixed.
In view of the definition of $\bar {v}$ in \eqref{bbb1}, in order to prove that $\bar v(\alpha) = 1$ holds, it is enough to show that for any $\delta > 0$, if $t > 0$ is sufficiently small, the problem \eqref{aaa} admits a feasible solution $y^t \in \Re^3$ with objective value at least $1-\delta$, i.e.,
\begin{equation}\label{a2}
(1+t) y_{1}^t + t y_{2}^t \ + t y_{3}^t  \geq 1 -\delta.
\end{equation}
This would imply $v(t\alpha, t) \geq 1-\delta$ for sufficiently small $t$ and thus $\bar{v}(\alpha) \geq 1-\delta$ holds. Then, since $\delta$ is arbitrary, we obtain $\bar{v}(\alpha) = 1$.

In view of this, we focus our efforts on establishing the existence of feasible solutions as in \eqref{a2}.
We will proceed by analyzing what conditions must a feasible solution to \eqref{aaa} satisfy.
We note that since \eqref{aaa} is strongly feasible, there are feasible solutions corresponding to positive definite matrices such that the objective value is arbitrarily close to the optimal value. 
Therefore, it is enough to focus on analyzing the $y$'s that are feasible to \eqref{aaa} and correspond to positive definite matrices.

So suppose that $y \in \Re^3$ corresponds to a positive definite solution to \eqref{aaa}.
First, we take the Shur complement with respect to the $(4,4)$-element. 
Then, positive definiteness is equivalent to 
\begin{eqnarray}   
&&\left(\begin{array}{ccc} 1+ t\alpha -y_1 -\frac{y_3^2}{t \alpha} & 0 & -y_3 \\
0 & t\alpha -y_2 & -y_1 \\ 
-y_3 & -y_1 & -y_3+t\alpha 
\end{array}\right) \succ  0. \label{bbb}
\end{eqnarray}

Next, we take the Shur complement of  the left hand side of \eqref{bbb} with respect to the $(1,1)$-element.  
Then, positive definiteness is equivalent to the following conditions
\begin{eqnarray*} 
&&1+ t\alpha -y_1 -\frac{y_3^2}{t \alpha} > 0,\quad
\left(\begin{array}{cc}  
 t\alpha -y_2 & -y_1 \\ 
 -y_1 & -y_3+t\alpha -\frac{y_3^2}{1+ t\alpha -y_1 -\frac{y_3^2}{t \alpha}}
\end{array}\right) \succ  0.
\end{eqnarray*}
This is equivalent to 
\begin{equation} \label{ccc}
\begin{array}{l}
1+ t\alpha -y_1 -\frac{y_3^2}{t \alpha} > 0, \\
t\alpha -y_2 > 0,\ \ \ (t\alpha -y_2)\left(-y_3+t\alpha -\frac{y_3^2}{1+ t\alpha -y_1 -\frac{y_3^2}{t \alpha}}\right) - y_1^2 >0 .
\end{array}
\end{equation}
Diving the {third} inequality by $\left(-y_3+t\alpha -\frac{y_3^2}{1+ t\alpha -y_1 -\frac{y_3^2}{t \alpha}}\right)$ 
which is ensured to be positive in view of \eqref{ccc},
we obtain that
\begin{equation} \label{ddd}
\begin{array}{l}
t\alpha - y_2 > 0, \ \ \ 
1+ t\alpha -y_1 -\frac{y_3^2}{t \alpha} > 0 ,\\ 
t\alpha -\frac{y_1^2}{ -y_3+t\alpha -\frac{y_3^2}{1+ t\alpha -y_1 -\frac{y_3^2}{t \alpha}}          }  > y_2,\ \ \ 
 -y_3+t\alpha -\frac{y_3^2}{1+ t\alpha -y_1 -\frac{y_3^2}{t \alpha}} > 0.
\end{array}
\end{equation}
is equivalent to  \eqref{ccc}.  Therefore, if $y \in \Re^3$ satisfies \eqref{ddd}, then $y$ is a feasible solution to 
\eqref{aaa} corresponding to a positive definite matrix.

Let $y_3 \coloneqq-\gamma\sqrt{\alpha t}$, and let $y_1 \coloneqq1-\gamma^2 - \zeta$, where $\gamma, \zeta$ are arbitrary positive numbers.

Then, we have
\[
1+t\alpha-y_1 -\frac{y_3^2}{t\alpha}= 1+ t\alpha  - y_1 - \gamma^2 = t\alpha+\zeta>0
\] 
and
\[
-y_3+t\alpha -\frac{y_3^2}{1+ t\alpha -y_1 -\frac{y_3^2}{t \alpha}}=
\gamma\sqrt{\alpha t}+\alpha t - \frac{\gamma^2\alpha t}{\alpha t + \zeta}.
\]
Furthermore, let
\[
y_2 \coloneqq t\alpha -\frac{y_1^2}{ -y_3                  +t\alpha -\frac{y_3^2}{1+ t\alpha -y_1 -\frac{y_3^2}{t \alpha}}} -\xi = 
t\alpha -\frac{(1-\gamma^2-\zeta)^2}{\gamma\sqrt{\alpha t}+\alpha t - \frac{\gamma^2\alpha t}{\alpha t + \zeta}}-\xi,
\]
where $\xi>0$ is an arbitrary positive number. 
Then $y \coloneqq  (y_1, y_2, y_3)$ defined as above constitutes a strict feasible solution to \eqref{aaa}
when $t>0$ is sufficiently small, since they satisfy \eqref{ddd} when $t>0$ is sufficiently small.  
Furthermore, the objective value
\begin{eqnarray*}
(1+t\alpha)y_1+t y_2 + ty_3&=&(1+t\alpha)(1-\gamma^2-\zeta)+t^2\alpha-\frac{t (1-\gamma^2-\zeta)^2}{\gamma\sqrt{\alpha t}
 +\alpha t  - \frac{\gamma^2\alpha t}{\alpha t + \zeta}} \\
 &&\ \ -t\xi -\gamma t\sqrt{\alpha t},
\end{eqnarray*}
which can be arbitrarily close to $1-\gamma^2-\zeta$, if $t > 0$ is taken sufficiently close to zero.  By taking $\gamma^2+\zeta =\delta/2$, we obtain a strict feasible solution 
to \eqref{aaa} satisfying \eqref{a2} for $t$ sufficiently close to 0, as we desired.
\end{proof}

\section{Concluding Discussion}\label{sec:conc}
In this paper, we analyzed the behavior of $\av(\theta)$
which was introduced in \cite{TLMO2019} to bridge the primal and dual optimal value in the presence of 
a nonzero duality gap. 
We assumed that {\bf P} and {\bf D} are feasible and they have nonzero duality gap, 
and showed that, surprisingly,   
$\av$ is a monotone bijective function
from $[0,\pi/2]$ to $[v({\bf D}),v({\bf P})]$
 if the singularity degrees of {\bf P} and {\bf D}
are both one, thus filling the duality gap completely.
However, we also produced an example showing that when the singularity degree is higher, then 
$\av$ can be {\em discontinuous} at $\theta=0$.
The study of deeper relations between the discontinuity of $\av$ and the singularity degree 
of {\bf P} and {\bf D} is an interesting topic for further research.
Another interesting direction of further research is to extend the results  to the case where 
either {\bf P} or {\bf D} is weakly infeasible (or  both!).  We note that there are several papers focused on weakly infeasible problems and their underlying structure \cite{waki_how_2012,lourenco_muramatsu_tsuchiya,pataki_touzov_an_echelon}. These works may provide clues on how to extend our results to the weakly infeasible case. 


\bibliographystyle{abbrvurl}
\bibliography{perturb_gap} 

{
\appendix
\section{Proof of item 2 of Proposition \ref{prop:sd1}}\label{app:proof}
In this appendix, we provide a proof of item 2 of Proposition~\ref{prop:sd1}.
Without loss of generality, 
we may assume that $A^1, \ldots, A^m$ are linearly independent. 
Let ${\cal V} \coloneqq \{\sum_{i=1}^m A^i y_i \mid y\in \Re^m\}.$
Since {\bf RD} is feasible, the subspace 
\[
{\cal W} \coloneqq 
\{W=\sum_{i=1}^m A^i y_i \mid y \in \Re^m, \ W_{12}=0,  W_{22}=0 \}
\]
of ${\cal V}$ is nonempty. Here we are using the convention that given $Z \in \S^n$, the matrices $Z_{11} \in \S^r, Z_{12} \in \Re^{r\times(n-r)}, Z_{22} \in \S^{n-r}$ denote the blocks of $Z$ according to the division given in Proposition~\ref{prop:sd1}.
We also recall that, in view of item~(1) of Proposition~\ref{prop:sd1}, there exists at least one $y$ for which $L_{12}(y) = 0$ and $L_{22}(y)=0$. This implies that $\cal L$ satisfies
\[
{\cal L}=\left\{ \left(\begin{array}{cc} 0 &  \sum_{i=1}^m A_{12}^i y_i \\
\sum_{i=1}^m (A_{12}^i)^T y_i &  \sum_{i=1}^m A_{22}^i y_i \end{array}\right)  \,\middle|\, y\in \Re^m \right\}
\]
and
\begin{equation}\label{Cbelongs}
\left(\begin{array}{cc} 0 & C_{12} \\
C_{12}^T  & C_{22} \end{array}\right) \in {\cal L}.
\end{equation}
In particular, $\cal L$ is a linear space. 
 We define $\cal Y$ as the subspace of $\S^n$ whose $(1,1)$-block is zero, i.e., $Y \in \S^n$ belongs to $\cal Y$ if and only if $Y_{11} = 0$.
Then, let $\pi: \S^n \rightarrow {\cal Y}$ denote the orthogonal projection onto ${\cal Y}$ so that the following holds for every $Y \in \S^n$:
\[
\pi\left(\left(\begin{array}{cc}Y_{11}& Y_{12}\\ Y_{12}^T & Y_{22}\end{array}\right)\right)=
\left(\begin{array}{cc} 0& Y_{12}\\ Y_{12}^T & Y_{22}\end{array}\right).
\]

Recall that our perturbation space is ${\cal T}\oplus{\cal L}$.  Given $S\in {\cal T}\oplus{\cal L}$,
$S$ is represented uniquely as  the sum of $S_{\cal T} \in {\cal T}$ and $S_{\cal L} \in {\cal L}$.
This decomposition is given as 
\[
S=
s_{11} I + 
\left(\begin{array}{cc} 0& S_{12}\\ S_{12}^T & S_{22} - s_{11} I_{22}\end{array}\right),\]
where $I = \begin{pmatrix}
I_{11} & 0\\
0 & I_{22}
\end{pmatrix}$. With that, we have
\begin{equation}\label{Sdcomp}
S_{\cal T} = s_{11}I\in {\cal T},
\  S_{\cal L} = \left(\begin{array}{cc} 0& S_{12}\\ S_{12}^T & S_{22} - s_{11} I_{22}\end{array}\right)\in {\cal L}.
\end{equation}

Now we are ready to proceed. 
For the proof of item 2 of Proposition \ref{prop:sd1}, we take a suitable basis of $\cal V$ as described in the following proposition.
\begin{proposition}\label{prop:bi}
There are matrices $B^1, \ldots, B^m \in \S^n$ with the following properties:
\begin{enumerate}
\item $\{B^1, \ldots, B^m\}$ forms a basis of $\cal V$.
\item $\{B^1, \ldots, B^k\}$ forms a basis of $\cal W$ ($k < m$), 
in particular, $B^i_{12}=0$ and $B^i_{22}=0$ for all $i=1, \ldots, k$.
\item Every element $V= \sum_{i=1}^m A^i y_i \in {\cal V} $ is written uniquely as 
\begin{align*}
V &=
\sum_{i=1}^k  B^i z_i +  
\sum_{i=k+1}^m  B^i z_i \\
&=
\left(\begin{array}{cc} \sum_{i=1}^k  B_{11}^i z_i & 0    \\
0 &  0 \end{array}\right) 
+
\left(\begin{array}{cc} \sum_{i=k+1}^m  B_{11}^i z_i & \sum_{i=k+1}^m  B_{12}^i z_i    \\
\sum_{i=k+1}^m  (B_{12}^i)^T z_i  & \sum_{i=k+1}^m  B_{22}^i z_i \end{array}\right).
 \end{align*}
\item There exists a nonsingular $m\times m$ matrix $D$ such that $B^i =\sum_{j=1}^m d_{ij} A^j$ $(i=1,\ldots,m)$ 
holds.  If $y,z \in \Re^m$ are such that $\sum_{i=1}^m A^i y_i =\sum_{j=1}^m B^j z_j  $, then $y^T = z^T D$.
\end{enumerate}
\end{proposition}
\begin{proof}
To obtain the basis $B^1, \ldots, B^m$ satisfying items 1 and 2, 
we first construct a basis of $\cal W$ and then expand it by adding independent elements from $\cal V$
until the chosen elements forms a basis of $\cal V$.  Item 3 is a direct consequence of items 1 and 2,
and item 4 follows from a standard argument in linear algebra 
since $\{A^1,\ldots A^m\}$ and $\{B^1,\ldots,B^m\}$ are two bases of the same linear space $\cal V$. 
\end{proof}

\begin{lemma}\label{unique}
With the $B^i$ as in Proposition~\ref{prop:bi}, define $\tau:\Re^{m-k}\to {\cal L}$, $\tau_{12}:\Re^{m-k}\to \Re^{r\times(n-r)}$ and $\tau_{22}:\Re^{m-k}  \to \S^{n-r}$ such that
\begin{eqnarray*}
&&\tau(z_{k+1}, \ldots, z_m) \nonumber\\ 
&&\ \ \coloneqq
\left(\begin{array}{cc} 0 & \tau_{12}(z_{k+1},\ldots,z_m)    \\
\tau_{12}(z_{k+1},\ldots,z_m)^T &  \tau_{22}(z_{k+1},\ldots,z_m)\end{array}\right) \nonumber \\
&&\ \ \coloneqq
\left(\begin{array}{cc} 0& \sum_{i=k+1}^m B^i_{12} z_i  \\
\sum_{i=k+1}^m B^i_{12} z_i &  \sum_{i=k+1}^m B^i_{22} z_i \end{array}\right) 
 = \pi\left(\sum_{i=k+1}^m B^i z_i\right).\ 
\end{eqnarray*}
Then, $\tau$ is a bijective linear map.
\end{lemma}
\begin{proof}
We first observe that $\tau$ is surjective. Recall that $\{B^{1}, \ldots,B^m\}$ is a basis of $\cal V$ and 
$\cal L=\pi({\cal V})$.  
Pick any element $X$ in $\cal L$.  Since $\cal L=\pi({\cal V})$, 
there exists $z\in \Re^m$ such that  $X=\pi(\sum_{i=1}^m B^i z_i)$.
Then it follows that
\[
X=\pi\left(\sum_{i=1}^m B^i z_i\right)=\pi\left(\sum_{i=k+1}^m B^i z_i\right) = \sum_{k+1}^m \pi(B^i) z_i = \tau(z^{k+1},\ldots,z^m),
\]
and this proves that $\tau$ is a surjection.

Next, we check that $\tau$ is injective by showing that its kernel is trivial. Suppose that $\tau(z_{k+1},\ldots, z_m) = 0$. 
That is, 
\[
\left(\begin{array}{cc} 0& \sum_{i=k+1}^m B^i_{12} z_i  \\
\sum_{i=k+1}^m B^i_{12} z_i &  \sum_{i=k+1}^m B^i_{22} z_i \end{array}\right)  = \begin{pmatrix}
0 & 0\\
0 & 0
\end{pmatrix}
\]
holds.
Therefore, $\sum _{i=k+1}^m B^{i}z_i \in \cal{W}$. Since $B^1, \ldots, B^k$ form a basis for $\cal{W}$, there are $(z_1,\ldots, z_k)$ such that \[\sum _{i=1}^k B^{i}z_i = \sum _{i=k+1}^m B^{i}z_i.\]
By the linear independence of the $B^i$, all the $z_i$ must be $0$, so $\tau$ is injective as claimed.
Thus, $\tau$ is surjective and injective, and hence $\tau$ is bijective, and the proof is complete.
\end{proof}

Now, we are ready to prove item 2, i.e.,
the optimal value function $w(S)$ of the following problem:
\[
{\bf RD}(S)\ \ 
\min_{y} b^Ty\ \ \hbox{subject\ to}\ L_{11}(y)+s_{11}I_{11} \succeq 0, \ L_{12}(y) = S_{12},\ L_{22}(y)+s_{11}I_{22}=S_{22},
\]
where $S\in {\cal T} \oplus {\cal L}$, is continuous at $S=0$.  We note that $s_{11} I_{11} = S_{11}$ holds for any
$S\in {\cal T} \oplus {\cal L}$.

To this end, we rewrite this problem in terms of $B^i, i=1, \ldots, m$ and $z$, to obtain that
\begin{equation}\label{Bz}
\begin{array}{lll}
\min_{z} b^TD^T z &\hbox{s.t.} & C_{11}+s_{11}I_{11} - \sum_{j=1}^m B^j_{11} z_j \succeq 0, \\
 & &C_{12}-\sum_{i=k+1}^m B^i_{12} z_i=S_{12},\\  
 & & C_{22}+s_{11}I_{22}-\sum_{i=k+1}^m B^i_{22} z_i=S_{22}.
 \end{array}
\end{equation}
Now, observe that the two equality constraints in \eqref{Bz} can be written as 
\begin{equation}\label{tau1}
\tau_{12}(z_{k+1},\ldots,z_m)=C_{12}-S_{12},\ \ \ \tau_{22}(z_{k+1},\ldots,z_m)=C_{22}-(S_{22}-s_{11}I_{22}).
\end{equation}
Since $\tau$ is a bijective mapping from $\Re^{m-k}$ to $\cal L$ as was shown in Lemma \ref{unique}, its inverse mapping is well-defined.
Let us denote by $\tau^{\rm inv}$ the inverse mapping of $\tau$.  
Slightly abusing the notation, we write $\tau^{\rm inv}(\widetilde L_{12}, \widetilde L_{22})$ as a function of $\widetilde L_{12}$ and $\widetilde L_{22}$ implicitly 
assuming that 
$\left(\begin{array}{cc} 0 & \widetilde L_{12}\\ \widetilde L_{12}^T & \widetilde L_{22} \end{array}\right) \in {\cal L}$.
Then, recalling
that 
$\left(\begin{array}{cc}0 & C_{12}\\ C_{12}^T & C_{22}\end{array}\right) \in \mathcal{L}$ and 
$\left(\begin{array}{cc}0 & S_{12}\\ S_{12}^T & S_{22}-s_{11}I\end{array}\right) \in \mathcal{L}$
 (see \eqref{Cbelongs}, \eqref{Sdcomp}), we see that  \eqref{tau1} implies that
\[
(z_{k+1}, \ldots, z_{m})=(\tau_{k+1}^{\rm inv}(C_{12}-S_{12},C_{22}-(S_{22}-s_{11}I_{22})), \ldots, \tau_m^{\rm inv}(C_{12}-S_{12},C_{22}-(S_{22}-s_{11}I_{22}))
\] 
must hold, that is, $(z_{k+1}, \ldots, z_{m})$ is determined uniquely by the perturbation $S$ in \eqref{Bz}.
Therefore, \eqref{Bz} is equivalent to the following semidefinite programs with respect to $(z_1, \ldots, z_k)$:
\begin{eqnarray}
\min_{(z_1,\ldots,z_k)}&& \sum_{i=1}^k(D b)_i z_i + \sum_{i=k+1}^m(Db)_i \tau_i^{\rm inv}(C_{12}-S_{12},C_{22}-(S_{22}-s_{11}I_{22})) \nonumber\\
\hbox{s.t.} &&
(C_{11}+s_{11}I_{11} - \sum_{j=k+1}^m B^j_{11}\tau_{j}^{\rm inv}(C_{12}-S_{12},C_{22}-(S_{22}-s_{11}I_{22})))\label{Dz} \\
&&\ \ \ - \sum_{j=1}^k B^j_{11} z_j \succeq 0.\nonumber
\end{eqnarray}
The optimal value function of \eqref{Dz} is is equal to $w(S)$.

Since $\tau^{\rm inv}$ is a linear mapping and  \eqref{Cbelongs} and \eqref{Sdcomp} holds, we have, for each $j =k+1, \ldots, m,$
\[
\tau_{j}^{\rm inv}(C_{12}-S_{12},C_{22}-(S_{22}-s_{11}I_{22}))=\tau_{j}^{\rm inv}(C_{12},C_{22})-\tau_{j}^{\rm inv}(S_{12},S_{22}-s_{11}I_{22}).
\]
Therefore, we may further rewrite \eqref{Dz} as
\begin{eqnarray}
&&
\min_{(z_1,\ldots,z_k)} \sum_{i=1}^k(D b)_i z_i + \sum_{i=k+1}^m(Db)_i \tau_i^{\rm inv}(C_{12},C_{22}) 
 - \sum_{i=k+1}^m(Db)_i \tau_i^{\rm inv}(S_{12},S_{22}-s_{11}I_{22}) \nonumber \\
&&\ \ \hbox{s.t.} \ 
(C_{11}+ \sum_{j=k+1}^m B^j_{11} \tau_{j}^{\rm inv}(C_{12},C_{22})) +s_{11}I_{11} \label{Ez}\\
&&\ \ \ \ \ \ \ \ -\sum_{j=k+1}^{m} B^j_{11} \tau_{j}^{\rm inv}(S_{12},S_{22}-s_{11}I_{22}) 
- \sum_{j=1}^k B^j_{11} z_j \succeq 0. \nonumber
\end{eqnarray}
Overall,  \eqref{Ez} is equivalent to ${\bf RD}(S)$ and  
the optimal value of \eqref{Ez} as a function of perturbation $S\in {\cal T}\oplus{\cal L}$ coincides with $w(S)$.  

Since the second and third term of the objective function is a constant and a linear function of $S$, 
respectively, 
in order to establish continuity of $w(S)$ at $S=0$, 
it is enough to show continuity of the optimal value function of  the semidefinite program
\begin{eqnarray}
&&
\min_{(z_1,\ldots,z_k)} \sum_{i=1}^k(D b)_i z_i  \nonumber \\
&&\ \ \ \hbox{s.t.} \ 
(C_{11}+ \sum_{j=k+1}^m B^j_{11} \tau_{j}^{\rm inv}(C_{12},C_{22})) +s_{11}I_{11} \label{Fz}\\
&&\ \ \ \ \ \ \ \ -\sum_{j=k+1}^{m} B^j_{11} \tau_{j}^{\rm inv}(S_{12},S_{22}-s_{11}I_{22}) 
- \sum_{j=1}^k B^j_{11} z_j \succeq 0. \nonumber
\end{eqnarray}
obtained by dropping the 
second and the third terms of the objective function in \eqref{Ez}.
The feasible region of \eqref{Fz} is the same as \eqref{Ez} and hence \eqref{Fz} is strongly feasible at $S=0$.
The perturbation term
\[
s_{11}I_{11} -\sum_{j=k+1}^{m} B^j_{11} \tau_{j}^{\rm inv}(S_{12},S_{22}-s_{11}I_{22})
\]
exists only at ``the constant part of the constraint'' in $\eqref{Fz}$ and this term vanishes at $S=0$.  
Therefore, we can directly apply Theorem 4.1.9 of \cite{sdp_handbook} to show that the optimal value function 
of \eqref{Fz} is continuous at $S=0$ and so is the optimal value function of \eqref{Ez}.
This completes the proof.

}
\end{document}